\documentclass[11pt]{article}
\usepackage{amsmath}
\usepackage{amsfonts}
\usepackage{amsthm}
\usepackage{amssymb}
\usepackage{color}

\newtheorem{theorem}{Theorem}[section]
\newtheorem{lemma}[theorem]{Lemma}

\newtheorem{observation}[theorem]{Observation}
\newtheorem{problem}[theorem]{Problem}

\newtheorem{example}[theorem]{Example}
\newtheorem{corollary}[theorem]{Corollary}

\theoremstyle{definition}
\newtheorem{definition}[theorem]{Definition}
\theoremstyle{remark}
\newtheorem{remark}[theorem]{Remark}

\newcommand\remove[1]{}

\def\cay{\hskip0.02cm{\rm Cay}\hskip0.01cm}

\def\N{{\mathbf{N}}}

\def\f2{\mathbb{F}_2}

\begin{document}

\title{\LARGE Metric dimensions of minor excluded graphs and minor exclusion in groups}

\author{Mikhail I.~Ostrovskii and David Rosenthal\footnote{The first-named author was supported by NSF
DMS-1201269. The second-named author was supported by the Simons
Foundation \#229577. The authors are very thankful to Florent
Baudier, Genady Grabarnik, Volodymyr Nekrashevych, Henrik Rueping and Andreas
Thom for useful discussions.}
\\
\\
Department of Mathematics and Computer Science\\
St. John's University\\
8000 Utopia Parkway\\
Queens, NY 11439\\
USA\\
Fax: (718) 990-1650\\
e-mails: {\tt ostrovsm@stjohns.edu}, {\tt rosenthd@stjohns.edu}}

\date{\today}
\maketitle



\noindent{\bf Abstract.} An infinite graph $\Gamma$ is
minor excluded if there is a finite graph that is not a minor of $\Gamma$.  We
prove that minor excluded graphs have finite Assouad-Nagata
dimension and study minor exclusion for Cayley graphs of finitely
generated groups. Our main results and observations are: (1) minor
exclusion is not a group property: it depends on the choice of
generating set; (2) a group with one end has a generating set for
which the Cayley graph is not minor excluded; (3) there are groups that are not minor excluded for any set of generators, like $\mathbb{Z}^3$; (4) minor exclusion is preserved under free products; and (5) virtually free groups are minor
excluded for any choice of finite generating set.
\medskip

\medskip

\medskip

\section{Introduction}
A finite graph $M$ is a {\it
minor} of a connected graph $\Gamma$ if there is a finite set
$\{V_i\}$ of pairwise-disjoint finite connected subgraphs of
$\Gamma$ (called {\it branch sets}) such that the set $\{V_i\}$ is
in one-to-one correspondence with the set of vertices $\{v_i\}$ of
$M$, and for every edge in $M$ between vertices $v_i$ and $v_j$ in
$M$, there is an edge in $\Gamma$ between the corresponding branch
sets $V_i$ and $V_j$. The graph $\Gamma$ is {\it minor excluded}
if there is a finite graph that is not a minor of $\Gamma$. Since
every finite graph is a subgraph of some complete graph, $\Gamma$
is minor excluded if and only if there exists a natural number $m$
such that the complete graph $K_m$ on $m$ vertices is not a minor
of $\Gamma$. Minor exclusion of groups is about minor exclusion of
Cayley graphs.

Minor exclusion plays an important role in graph theory. The
well-known Kuratowski Theorem states that a finite graph is planar
if and only if the complete graph on five vertices, $K_5$, and the
complete bipartite graph on six vertices, $K_{3,3}$, are excluded
as minors. If one defines an infinite graph to be planar
provided there is an embedding of the graph into $\mathbb{R}^2$, then
Kuratowski's Theorem generalizes to infinite graphs as well
\cite{DS54} (the theory of planarity of infinite graphs is
somewhat different if the accumulation points of vertices are not
allowed; see \cite{Geo14} and references therein). A finitely
generated group $G$ is called planar if there exists a finite
symmetric generating set $S$ of $G$ for which $\cay(G,S)$ is
planar. There is extensive literature on planar groups
(see~\cite{Dro06,DSS98,Geo14,Lev82,Ren05,ZVC80}), and, since
planar graphs can be described in terms of minor exclusion, this
study is a part of the considered subject matter. As for the study
of minor exclusion for groups in general, we found only one other
paper~\cite{AC04}.

Embedding groups into Banach spaces is a very useful
tool for studying groups with respect to applications in topology, most notably to the Novikov Conjecture (see, for example,~\cite{NY12}). It is known from the works \cite{KPR93} and \cite{Rao99} that problems about embeddability of graphs into Banach spaces are closely related to the theory of minors in graphs.
Related to embedding questions is the study of
various large-scale notions of dimension for groups, such as {\it asymptotic dimension} and {\it Assouad-Nagata dimension} (see
\cite{Gro93}, where many of such notions are originated, and
\cite{BS07}).

Our first goal in this paper is to show that if an infinite graph $\Gamma$ is minor excluded, then $\Gamma$ has finite Assouad-Nagata dimension~(Theorem \ref{T:MinExclANDim}). This also implies that the asymptotic dimension of an infinite minor excluded graph is finite. After that we study the notion of minor exclusion for groups. Our main results and observations on minor exclusion for groups are:

\begin{itemize}

\item If $G$ is a finitely generated group with one end, then there is
a generating set $S$ in $G$ such that $\cay(G,S)$ is not minor
excluded (Theorem \ref{T:OneEndInfInd}). Since $\mathbb{Z}^2$
has one end and its standard Cayley graph is planar (and thus
minor excluded), this result implies that the minor exclusion of
$\cay(G,S)$ is not a group property, in the sense that it depends
on the generating set $S$. Actually, the fact that minor exclusion
is not a group property is much simpler than Theorem~\ref{T:OneEndInfInd} (see Example \ref{E:NotMEinZ2}).

\item There is a large class of groups whose Cayley graphs are not
minor excluded for any choice of a generating set (Section
\ref{S:Never}). This class includes all groups containing
$\mathbb{Z}^3$ as a subgroup. Theorem
\ref{T:MinExclANDim} also has a corollary of this type (see Remark
\ref{R:CombWthNS}).

\item A virtually free group is minor excluded for any choice of
generating set (Theorem \ref{T:Free}).

\item Minor exclusion is preserved under free products (Theorem
\ref{T:FreeProd}). This result generalizes a result of
Arzhantseva and Cherix that planarity is preserved under free
products.

\end{itemize}

\section{Minor exclusion and metric dimensions}
An infinite graph $\Gamma$ is {\it
connected} if there is a (finite) path between any two vertices of
$\Gamma$. A finite graph $M$ is a {\it minor} of a connected graph
$\Gamma$ if there is a finite set $\{V_i\}$ of pairwise-disjoint
finite connected subgraphs of $\Gamma$ (called {\it branch sets})
such that the set $\{V_i\}$ is in one-to-one correspondence with
the set of vertices $\{v_i\}$ of $M$, and for every edge in $M$
between vertices $v_i$ and $v_j$ in $M$, there is an edge in
$\Gamma$ between the corresponding branch sets $V_i$ and $V_j$.
The graph $\Gamma$ is {\it minor excluded} if
and only if there exists a natural number $m$ such that the
complete graph $K_m$ on $m$ vertices is not a minor of $\Gamma$.
(Our graph theory terminology and notation mostly follows
\cite{Die00}.)

\begin{definition}[\cite{Ass82, LS05}]
Let $X$ be a metric space and $d\in\N$. The {\it Assouad-Nagata dimension} (or just {\it Nagata dimension}) of $X$ is at most $d$ if there exists a $\gamma\in(0,\infty)$
such that for every $s>0$ there exists a cover $\mathcal{U}$ of $X$ with $s$-multiplicity at most $d+1$ (i.e., each closed ball of radius $s$ in $X$ intersects at most $d+1$ elements of $\mathcal{U}$), whose elements each have diameter at most $\gamma \cdot s$.
\end{definition}

\begin{theorem}\label{T:MinExclANDim}
If $\Gamma$ is a connected graph with finite degrees excluding the
complete graph $K_m$ as a minor, then $\Gamma$ has Assouad-Nagata
dimension at most $4^m-1$.
\end{theorem}

\begin{remark}\label{R:asdim}
An immediate corollary to Theorem~\ref{T:MinExclANDim} is that the {\it asymptotic dimension} of a minor excluded graph is finite. This is because, by definition, the asymptotic dimension of $X$ is at most $d$ if for every $s>0$ there exists a cover $\mathcal{U}$ of $X$ with $s$-multiplicity at most $d+1$, whose elements have uniformly bounded diameter. Thus, the asymptotic dimension of $X$ is bounded above by the Assouad-Nagata dimension of $X$.
\end{remark}

\begin{remark}\label{R:OnOst09} In \cite{Ost09} it was shown that minor excluded connected
infinite graphs with finite degrees of vertices admit coarse
embeddings into a Hilbert space. Combining Theorem~\ref{T:MinExclANDim} with the results of Naor and Silberman
\cite{NS11} we get the following strengthening of the result of
\cite{Ost09}: any snowflaking of a minor excluded connected graph
with finite degrees admits a bilipschitz embedding into a Hilbert
space ({\it snowflaking} means passing from the
metric space $(X,d)$ to the metric space $(X,d^\theta)$, where
$\theta\in(0,1)$).
\end{remark}

\begin{proof}[Proof of Theorem \ref{T:MinExclANDim}]  Observe that for graphs (considered as vertex sets
with the shortest path metric, where each edge has length $1$)
it suffices to consider $s\in\mathbb{N}$ (although
restricting $s$ to integers will cause us to increase $\gamma$). So
we need to prove that there exists $0<\gamma<\infty$ (which may
depend on $m$) such that for every $s\in\mathbb{N}$ there is a
cover of $\Gamma$ with sets of diameter at most $\gamma\cdot s$ and
$s$-multiplicity at most $4^m$. To find such a cover we construct $4^m$ ``partitions" of $\Gamma$, following
the approach of \cite{KPR93} with a slight modification.
Elements of the partitions will be defined as connected components of the graph
obtained from $\Gamma$ after removing $m$ sets of edges
$\{F_i\}_{i=1}^m$, constructed as follows.

Enumerate all of the vertices of $\Gamma$ as $\{v_1,v_2,\dots\}$
(it is clear that connected graphs with finite degrees are
countable), and let $R=s+3$. Let $\delta_1$ be one of the elements
of the set $\{0,R,2R,3R\}$ and define $F_1$ to be the set of all
edges that join vertices $u$ satisfying $d(u,v_1)=4Rj+\delta_1$,
for some $j\in\mathbb{N}$, to vertices $w$ satisfying
$d(w,v_1)=4Rj+\delta_1+1$ for the same $j$. Now delete $F_1$ from
the edge set of $\Gamma$. It is clear that unless $d(u,v_1)\le
4R+\delta_1$ for all $u\in V(\Gamma)$, we get a disconnected
graph. To construct $F_2$, pick $\delta_2\in \{0,R,2R,3R\}$
independently from the choice of $\delta_1$. In each component $C$
of $\Gamma\smallsetminus F_1$ choose a vertex $v_{i(C)}$ with
smallest index $i(C)$ over all vertices of $C$. Consider the set
of all edges that join vertices $u$ satisfying
$d_C(u,v_{i(C)})=4Rj+\delta_2$, for some $j\in\mathbb{N}$, to
vertices $w$ satisfying $d_C(w,v_{i(C)})=4Rj+\delta_2+1$, where
$d_C$ is the path-length metric (or shortest path distance) of
$C$. Define $F_2$ as the union of such edge sets over all
components $C$ of the graph $\Gamma\smallsetminus F_1$ obtained
from $\Gamma$ after the deletion of edges of $F_1$. To construct
$F_3$, pick $\delta_3\in \{0,R,2R,3R\}$ independently from the
choice of $\delta_1$ and $\delta_2$. Repeat the procedure from the
previous paragraph for each component of the graph obtained from
$\Gamma$ after the removal of $F_1\cup F_2$. Continue in the
obvious way, making a series of $m$ cuts of this type. We call the
vertex sets of the connected components of $\Gamma \smallsetminus
\cup_{i=1}^mF_i$ {\it clusters}; this set of clusters yields a
partition $\mathcal{P}$ of $\Gamma$. Construct such
partitions for all possible choices of $\delta_1,\dots,\delta_m$.
We get $4^m$ different partitions
$\{\mathcal{P}_p\}_{p=1}^{4^m}$ of $\Gamma$.

Define a cover $\mathcal{U}$ of $\Gamma$ as follows. For each
cluster $T$ in some partition $\mathcal{P}_p$, let $U_T$ be the
set of all vertices of $T$ satisfying the following conditions:
(1) they have distance at least $s+1$ from the ends of the edges
of $F_1$ in the original distance of $\Gamma$; and (2) for each
integer $k$, $2\leq k \leq m$, they have distance at least $s+1$
from the ends of edges of $F_k$ in the shortest-path metric
for the graph $\Gamma \smallsetminus\cup_{i=1}^{k-1}F_i$.

The set $\mathcal{U}=\{U_T\}$, over all clusters $T$ of all partitions
$\mathcal{P}_p$, must form a cover. To verify this, let $w$ be a vertex
in $\Gamma$. We choose $\delta_1,\dots,\delta_m$ so that $w$ will be contained in $U_T$, where $T$ is the
cluster of the partition corresponding to
$\delta_1,\dots,\delta_m$ that contains $w$. This can be done in the
following way: select $\delta_1\in\{0,R,2R,3R\}$ so
that $d_\Gamma(w,v_1)-\delta_1$ modulo $4R$ is between $R$ and $3R$. After
removing the corresponding $F_1$, the vertex $w$ lands in one of
the components, $C$, of the graph $\Gamma\smallsetminus F_1$. Let
$v_{i(C)}$ be the vertex of this component with smallest
index. Select $\delta_2\in\{0,R,2R,3R\}$ so that
$d_C(w,v_{i(C)})-\delta_2$ modulo $4R$ is between $R$ and $3R$. Continue in the obvious way. It is straightforward to check, since $R=s+3$,
that $w$ satisfies conditions (1) and (2) above for the obtained
$F_1,\dots,F_m$.

Klein, Plotkin, and Rao \cite{KPR93} proved that the assumption
that $\Gamma$ does not have $K_m$ as a minor implies that the diameter
of any cluster in any $\mathcal{P}_p$ is at most $\alpha(m) \cdot R$,
where $\alpha(m)\in(0,\infty)$ depends only on $m$. Therefore, every element of $\mathcal{U}$ will have diameter at most $\gamma \cdot s$ if we choose $\gamma=4\cdot \alpha(m)$.

It remains to show that the $s$-multiplicity of $\mathcal{U}$ is
at most $4^m$. Since each $x$ is inside at most $4^m$ different
clusters $T$ (one cluster from each partition), this would follow
from the claim that for a given element $U_T$ of the cover, the
ball $B(x,s)$ can intersect $U_T$ only if
$x\in T$. To prove this claim, it suffices to establish the
following statement by induction on $k$:

If $x$ is separated from $T$ by the edge cut $\cup_{i=1}^k F_i$,
but not by $\cup_{i=1}^{k-1} F_i$ (the latter set is assumed to be
empty for $k=1$), then $x$ has distance at least $s$ from $U_T$.

For $k=1$ this is obvious, by item (1) in the definition of $U_T$.
Let $k=2$. Assume that there is a path of length at most $s$
joining $x$ and $U_T$. If this path does not use any edges of
$F_1$, then the path-length distance in $\Gamma\smallsetminus F_1$
between $x$ and $U_T$ is also at most $s$. But this contradicts
item (2) in the definition of $U_T$.
If this path uses an edge of $F_1$, we get a
contradiction with what we proved for $k=1$. The inductive step is the same as in the proof of the case $k=2$.
\end{proof}

\begin{remark}\label{R:CombWthNS} See \cite{NY12} for the definition of ``compression". Remark \ref{R:OnOst09} shows that Theorem~\ref{T:MinExclANDim} in combination with the results of Naor and
Silberman \cite{NS11} implies that minor excluded groups have
compression $1$.  Therefore, the Cayley graph of a group that does
not have compression $1$ is not minor excluded with respect to an
arbitrary set of generators.
\end{remark}

\begin{remark} It is worth mentioning that
estimates from \cite{KPR93} for $\alpha(r)$ were improved in
\cite{FT03} (see also a presentation of results of \cite{FT03},
\cite{KPR93}, and \cite{Rao99} in \cite[Section 3.2]{Ost13}).
\end{remark}

\section{Minor exclusion for groups}
Let $G$ be a finitely generated group, and let $S$ be a finite
generating set for $G$. Throughout this paper, we assume that $S$
does not contain the identity element of $G$ and that $S$ is {\it
symmetric}, i.e., $s\in S$ if and only if $s^{-1}\in S$. To study
minor exclusion of $G$, we will always use the {\it
right-invariant Cayley graph}, $\cay(G,S)$, i.e., the graph
with vertex set $G$ and edge set defined by the condition: $uv$ is
an edge between $u,v\in G$ if and only if $u=sv$ for some $s\in
S$. The group $G$ acts on the right of this graph.  (Our
references for group theory are \cite{Mei08} and \cite{Rob96}.)

When considering minor exclusion for groups, one quickly realizes, as Example~\ref{E:NotMEinZ2} shows, that minor exclusion is not a group property. That is, it depends not only on the group, but also on the choice of a generating set.

\begin{example}\label{E:NotMEinZ2} The Cayley graph $\cay(\mathbb{Z}^2, S_1)$,
 where $S_1=\{(\pm 1,0),~ (0,\pm1)\}$, is minor excluded by the well-known
Kuratowski theorem, since it is a planar graph. On the other hand,
the Cayley graph of $\mathbb{Z}^2$ with respect to the, slightly
bigger, set of generators $S_2=\{(\pm 1,0), (\pm2,0), (0,\pm1)\}$,
is not minor excluded.
\end{example}

\begin{proof} Let $m\in\mathbb{N}$ be given. Construct branch sets of a $K_m$-minor in $\cay(\mathbb{Z}^2, S_2)$ as
follows:
\[V_1=\big\{(2,1), \stackrel{(even,1)}{\dots}, (2m,1)\big\}\]
\[V_2=\big\{(3,1), (3,2), (4,2), \stackrel{(even,2)}{\dots}, (2m,2)\big\}\]
\[\vdots \]
\[V_k=\big\{(2k-1,1),(2k-1,2),(2k-1,3),\dots,(2k-1,k),(2k,k),\stackrel{(even,k)}{\dots}, (2m,k)\big\}\]
\[\vdots\]
\[V_m=\big\{(2m-1,1),\dots,(2m-1,m), (2m,m)\big\}.\]
It is straightforward to check that each $V_k$ is the vertex set of a connected subgraph of $\cay(\mathbb{Z}^2, S_2)$ and that any two of these vertex sets are joined by an edge.
\end{proof}

\subsection{Groups that are not minor excluded for any choice of generators}\label{S:Never}
It is interesting that if the group $\mathbb{Z}^2$ is increased
even ``slightly'', then we get a group that is not minor excluded
for any set of
generators. We mean the following result.

\begin{lemma}\label{L:Z2plusCyclic}
Let $C$ be a nontrivial cyclic group. Then $\mathbb{Z}^2 \times C$ is not minor excluded for any set of generators.
\end{lemma}

\begin{proof}
Let $S$ be a finite symmetric generating set for $\mathbb{Z}^2 \times C$. Since $\mathbb{Z}^2 \times C$ is abelian, there are distinct generators $s_1, s_2, s_3 \in S$ such that the subgroup $\langle s_1, s_2\rangle$ generated by $s_1$ and $s_2$ is isomorphic to $\mathbb{Z}^2$ and $s_3 \notin \langle s_1, s_2\rangle$.
Let $m\in\mathbb{N}$ be given and construct branch sets of
$\cay(\mathbb{Z}^2\times C, S)$ as follows:
\[V_1=\big\{s_1, 2s_1, 3s_1, \dots, ms_1\big\}\]
\[V_2=\big\{2s_1+s_3, 2s_1+s_2+s_3, 2s_1+s_2, \dots, ms_1+s_2 \big\}\]
\[\vdots\]
\[V_k=\big\{ks_1+js_2+s_3 \;:\; 0\leq j \leq k-1 \big\} \cup \big\{ls_1+(k-1)s_2 \;:\; k \leq l \leq m \big\}\]
\[\vdots\]
\[V_m=\big\{ms_1+s_3, ms_1+s_2+s_3, \dots, ms_1+(m-1)s_2+s_3, ms_1+(m-1)s_2 \big\}.\]
It is straightforward to check that each $V_k$ is the vertex set of a connected subgraph of $\cay(\mathbb{Z}^2\times C, S)$ and that any two of these vertex sets are joined by an edge. Therefore, we have constructed a $K_m$ minor in $\cay(\mathbb{Z}^2\times C, S)$ for every $m\in\mathbb{N}$; that is, $\cay(\mathbb{Z}^2\times C, S)$ is not minor excluded.
\end{proof}

\begin{remark}
The group $\mathbb{Z}^2\times \mathbb{Z}_2$ is in a sense the
smallest group satisfying the assumptions of Lemma
\ref{L:Z2plusCyclic}. It is worth noting that when $C$ is a
finite group, $\mathbb{Z}^2\times C$ is quasi-isometric to
$\mathbb{Z}^2$ (see \cite[p.~138]{BH99} for the definition).
\end{remark}

The following is a generalization of Babai's result that every subgroup of a planar group is planar~\cite{Bab77}.

\begin{theorem}\label{T:subgroupsNever}
Let $G$ be a finitely generated group containing a finitely generated subgroup $H$. If $H$ is not minor excluded for any set of generators, then $G$ is not minor excluded for any set of generators.
\end{theorem}

\begin{proof}
Let $X=\cay(G, S)$, where $S$ is a finite symmetric generating set for $G$. We must show that there is a $K_m$-minor of $X$ for every $m\in \mathbb{N}$.

We begin by recalling Babai's construction of a Cayley graph $X'$ for the subgroup
$H$~\cite{Bab77}. There is a
connected subgraph $T$ of $X$ that contains precisely one vertex
for each $H$-orbit in $V(X)$. This means that if $h\in H
\smallsetminus \{e\}$, then $V(T\cdot h) \cap V(T)=\emptyset$ and
$\bigcup_{h \in H} V(T\cdot h)=V(X)$. Then $X'$ is obtained by
collapsing each $T\cdot h$. That is, the vertex set of $X'$ is
$V(X')=\{T\cdot h\,:\,h \in H \}$ and there is one edge in $X'$
between the vertices $T\cdot h$ and $T\cdot h'$ if there is an
edge in $X$ connecting the subsets $T\cdot h$ and $T\cdot h'$.
Thus, $X'$ is connected (since $X$ is connected), $H$ acts on
$X'$, and the $H$-action is free and transitive on $V(X')$.
Therefore, $X'$ is a Cayley graph for $H$.

The generating set for
$H$ corresponding to $X'$ is the set $S'$ of all $k \in H$ such
that there is an edge in $X'$ between $T$ and $T\cdot k$. Note
that $S'$ might be infinite. Since $H$ is finitely generated,
however, there must be a finite subset $S''\subset S'$ that
generates $H$. Under the identification of $X'$ with $\cay(H,
S')$, let $X''$ be the subgraph of $X'$ corresponding to the
subgraph $\cay(H, S'')\subset \cay(H, S')$.
By assumption, $H$ is not minor excluded with respect to $S''$. Thus, for every $m\in \mathbb{N}$, there is a $K_m$-minor of $X''$. It is clear from the construction that a $K_m$-minor of $X''$ yields a $K_m$-minor of $X$. In other words, $G$ is not minor excluded with respect to $S$.
\end{proof}

By Lemma~\ref{L:Z2plusCyclic}, we have the following corollary to Theorem~\ref{T:subgroupsNever}.

\begin{corollary}\label{C:ContZ2plusCyclic}
Let $G$ be a finitely generated group that contains $\mathbb{Z}^2 \times C$ as a subgroup, where $C$ is a nontrivial cyclic group. Then the Cayley graph of $G$ is not minor excluded for any set of generators.
\end{corollary}

\subsection{Minor exclusion for groups containing finitely generated free subgroups of finite index}\label{S:VirtFree}

Next we consider groups that contain a finitely generated free group as a subgroup of finite index. Let $\mathbb{F}_n$ denote the free group on $n$ generators.

\begin{theorem}\label{T:Free} Let $G$ be a group containing $\mathbb{F}_n$ as a subgroup of finite index and
let $S$ be a finite symmetric generating set for $G$. Then
$\cay(G, S)$ is minor excluded.
\end{theorem}

\begin{proof} Let $\{x_1,\dots,x_n\}\subset G$ be the basis of $\mathbb{F}_n$, and let
$\{g_1=1,g_2,\dots,g_k\}$ be representatives of the left cosets of
$\mathbb{F}_n$ in $G$. As a set, $G$ may be identified with the
Cartesian product $G=\{g_1,\dots,g_k\}\times \mathbb{F}_n$ since
each $g\in G$ can be uniquely represented as $g=g_if$, where
$i\in\{1,\dots,k\}$ and $f\in\mathbb{F}_n$. Let
$O=\{x_1,\dots,x_n,x_1^{-1},\dots,x_n^{-1}\}$ and let
$\cay(\mathbb{F}_n,O)$ be the corresponding Cayley graph; denote
the associated distance function for this graph by
$d_{\mathbb{F}_n}$. To prove the theorem, assume the contrary.
That is, assume that the graph $\cay(G,S)$ has a $K_m$-minor for
every $m\in\mathbb{N}$.

\begin{lemma}\label{L:BetwComp} Let $e$ be an edge in $\cay(\mathbb{F}_n,O)$ and let
$A(e)$ and $B(e)$ be the vertex sets of the connected components
of $\cay(\mathbb{F}_n,O)\smallsetminus \{e\}$. The number of edges in $\cay(G,S)$
connecting $\{g_1,\dots,g_k\}\times A(e)$ and
$\{g_1,\dots,g_k\}\times B(e)$ is finite and bounded from above
independently of the choice of $e$.
\end{lemma}

\begin{proof}
The lemma is an immediate consequence of the following claim. If $u,v\in G$ are adjacent in $\cay(G,S)$, and $u=g_{i(1)}f_1$ and $v=g_{i(2)}f_2$, where $i(1),i(2)\in\{1,\dots,k\}$ and $f_1,f_2\in\mathbb{F}_n$,
then $d_{\mathbb{F}_n}(f_1,f_2)\le M$, where $M\in\mathbb{N}$
depends on $G$, $S$, and the choice of $\{g_1,\dots,g_k\}$ (but
not on the choice of $u$ and $v$). To verify this claim, note that there is an $s\in S$ such that $g_{i(2)}f_2=sg_{i(1)}f_1$, since $u$ and $v$ are adjacent in $\cay(G,S)$. Also, for
each $i\in\{1,\dots,k\}$ and $s\in S$, we have
$sg_i=g_{j(i,s)}f(i,s)$, for some $j(i,s)\in\{1,\dots,k\}$ and
$f(i,s)\in \mathbb{F}_n$. This implies that $f_2=f(i(1),s)f_1$.
Therefore, $d_{\mathbb{F}_n}(f_1,f_2)\le\max_{s\in
S}\max_{i\in\{1,\dots,k\}}d_{\mathbb{F}_n}(f(i,s),1).$ This
maximum is over a finite set, so the claim follows.
\end{proof}

Let $D$
denote the upper bound obtained in Lemma~\ref{L:BetwComp}. Given a
$K_m$-minor in $\cay(G, S)$, the number of edges of $\cay(G, S)$
connecting $\{g_1,\dots,g_k\}\times A(e)$ and
$\{g_1,\dots,g_k\}\times B(e)$ must be at least $R(e)+k_A(e)\cdot
k_B(e)$, where $R(e)$ is the number of branch sets of the
$K_m$-minor intersecting both of the sets (such branch sets are
said to {\it cross} $e$), $k_A(e)$ is the number of branch sets
that are completely contained in $\{g_1,\dots,g_k\}\times A(e)$,
and $k_B(e)$ is the number of branch sets completely contained in
$\{g_1,\dots,g_k\}\times B(e)$. We will show that, if $m$ is large
enough, the inequality $D\ge R(e)+k_A(e)\cdot k_B(e)$ leads to a
contradiction for a suitably chosen edge $e$.

Choose $m>\max \{6nD,3k \}$. Then
the following two conditions are satisfied:
\begin{enumerate}
\item[(1)]
$\frac1{2n-1}\left(\frac{2m}3-\frac{m}{6n}-k\right)>\frac{m}{6n}$;
\item[(2)] $\frac{m}{6n}>D$.
\end{enumerate}
By condition~(2), for any edge $e$ in $\cay(\mathbb{F}_n,O)$ the number of branch sets, $R(e)$, that cross both
$\{g_1,\dots,g_k\}\times A(e)$ and $\{g_1,\dots,g_k\}\times B(e)$
is less than $\frac{m}{6n}$. Therefore, because of condition~(2), we can establish a contradiction with the inequality $D\ge R(e)+k_A(e)\cdot k_B(e)$ if we find an edge for which $k_A(e)\cdot k_B(e)\ge\frac{m}{6n}$. We do this by showing that there exists an edge for which both $k_A(e)$ and $k_B(e)$ are positive and one of them is greater than $\frac{m}{6n}$.

Pick an edge $e_0$ in $\cay(\mathbb{F}_n,O)$ and assume that
$k_A(e_0)\ge k_B(e_0)$. If $k_B(e_0)\ge \frac{m}{6n}$, then we are
done. So assume that $k_B(e_0)<\frac{m}{6n}$. Let $e_1,\dots,
e_{2n-1}$ be the edges that have a common vertex, $f$, with $e_0$
and are contained in $\{g_1,\dots,g_k\}\times A(e_0)$. Let
$A(e_1),\dots,A(e_{2n-1})$ denote the vertex sets of the
components of $\cay(\mathbb{F}_n,O)\smallsetminus
\{e_1\},\dots,\cay(\mathbb{F}_n,O)\smallsetminus \{e_{2n-1}\}$,
respectively, that do not contain $e_0$. As noted above,
$R(e_{i})<\frac{m}{6n}$ for every $i$. That is, each of
$e_0,e_1,\dots, e_{2n-1}$ is crossed by fewer than $\frac{m}{6n}$
branch sets. Thus, there are more than $\frac{2m}{3}$ branch sets
that do not cross any of the edges $e_0,e_1,\dots, e_{2n-1}$. On
the other hand, since there are at most $k$ branch sets in the set
$\{g_1,\dots,g_k\}\times \{f\}$, the number of branch sets that do
not cross any of the edges $e_0,e_1,\dots, e_{2n-1}$ is at most
$$\left(\sum_{i=1}^{2n-1} k_A(e_i)\right)+k_B(e_0)+k < \left(\sum_{i=1}^{2n-1} k_A(e_i)\right)+ \frac{m}{6n}+k.$$
Therefore, $\big(\sum_{i=1}^{2n-1} k_A(e_i)\big)+ \frac{m}{6n}+k > \frac{2m}{3}$.
Applying (1), this implies that
at least one of the sets $\{g_1,\dots,g_k\}\times A(e_1),\dots,
\{g_1,\dots,g_k\}\times A(e_{2n-1})$ contains more than
$\frac{m}{6n}$ branch sets. That is, $k_A(e_i)>\frac{m}{6n}$ for some $i$, $1\leq i \leq 2n-1$. If $k_B(e_i)>0$, then we are done. If $k_B(e_i)=0$, then $k_B(e_i)<\frac{m}{6n}$ and we can repeat the argument. But the argument cannot be repeated infinitely many times, since the branch sets are finite and there are only finitely many of them. Therefore, we eventually find an edge $e$ with $k_A(e)>\frac{m}{6n}$ and $k_B(e)>0$, as claimed.
\end{proof}

\subsection{Stability of minor exclusion with respect to free products}

Recall that a {\it cut-vertex} in a graph is a vertex whose
deletion increases the number of components. The following theorem is a generalization of \cite[Theorem 3(i)]{AC04}, which states that $\cay(G*H,S\cup T)$ is
planar if both $\cay(G,S)$ and $\cay(H,T)$ are planar.

\begin{theorem}\label{T:FreeProd} Let $G$ and $H$ be finitely generated groups with generating sets $S$ and $T$, respectively. Let $M$ be a finite connected graph
that does not contain cut-vertices. If $M$ is a minor of
$\cay(G*H,S\cup T)$, then $M$ is a minor of either $\cay(G,S)$ or
$\cay(H,T)$.
\end{theorem}

\begin{proof}[Proof of Theorem \ref{T:FreeProd}] Let $V_1,\dots,V_m$ be
branch sets of an $M$-minor in $\cay(G*H,S\cup T)$. Recall that each element of $G*H$ is represented by a unique reduced word and nontrivial reduced words are products of the form $x_n\cdots x_1$, where $x_n,\dots,x_1$ are non-identity
elements alternating between $G$ and $H$ and $x_1$ can be either in
$G$ or in $H$ (see~\cite[p.~169]{Rob96} if needed). Recall that
two vertices $u,v$ in $\cay(G*H,S\cup T)$ are adjacent if and only if $u=sv$ for some $s\in S \cup T$. This immediately implies the following observation.

\begin{observation}\label{O:PthsFrPr} Let $x_n\cdots x_1$
and $y_t\cdots y_1$ be reduced words in $\cay(G*H,S\cup T)$ and
let $x_{k-1}\cdots x_1=y_{k-1}\cdots y_1$ denote their largest
common word, which can be empty. Then each path connecting
$x_n\cdots x_1$ and $y_t\cdots y_1$ passes through $x_{k}\cdots
x_1$ if $n\geq k$, and through $y_{k}\cdots y_1$ if $t\geq k$.
If $x_1\in G$ and $y_1\in H$, or vise versa, then the path passes through $1$.
\end{observation}

We shall use the notation $Hx_n\cdots x_1=\{x_{n+1}\cdot x_n\cdots
x_1\,:\, x_{n+1}\in H\}$, where $x_n,x_{n-2},\dots$ are fixed
elements of $G$ and $x_{n-1},x_{n-3},\dots$ are fixed elements of
$H$. Similarly, we define $Gy_m\cdots y_1$. Combining our assumptions with Menger's theorem~\cite[Section 3.3]{Die00}, it follows that
there exist two vertices $u_1\in V_1$ and $u_2\in V_2$ such that there
are two disjoint $u_1u_2$-paths in $\cay(G*H,S\cup T)$. Using
Observation~\ref{O:PthsFrPr}, we get that $u_1$ and $u_2$ are
either in the same set of the form $Gy_r\cdots y_1$, or in the same
set of the form $Hx_m\cdots x_1$. Without loss of generality, we can assume that $u_1$ and $u_2$ are in $Gy_r\cdots y_1$.
Since $\cay(G*H,S\cup T)$ is right-invariant, multiplying all
elements of $G*H$ on the right by $(y_r\cdots y_1)^{-1}$,
we may assume that $u_1$ and $u_2$ are in $G$. We claim that $V_1\cap G, \dots, V_m\cap G$ are the branch sets of an $M$-minor in $\cay(G,S)$.

The assumption that $M$ does not have cut-vertices implies that for each $i$, $3\leq i \leq m$, there
is a vertex $u_i \in V_i$ for which there are two disjoint paths, one
that is a $u_1u_i$-path and one that is a $u_2u_i$-path (this follows from~\cite[Corollary 3.3.3]{Die00} with $a=v_i$ and
$B=\{v_1,v_2\}$, where $\{v_i\}$ consists of vertices of $M$ corresponding
to the branch sets $V_1,\dots,V_m$). On the other hand, Observation~\ref{O:PthsFrPr} implies that two paths from different vertices of
$G$ to a vertex not in $G$ cannot be disjoint. Therefore,
$V_i\cap G\ne\emptyset$, for every $i$. Next, notice that for any two elements in $u,v\in V_i\cap G$ there is a $uv$-path $P$ in $\cay(G*H,S\cup T)$, all
of whose vertices are contained in $V_i$. Since a path does not
pass through any elements repeatedly, $P$ cannot leave $G$. Thus, the sets $V_i\cap G$ are connected in $\cay(G,S)$ for each $i$.

To complete the proof, we show that if there is an edge
in $M$ joining the vertices corresponding to $V_i$ and $V_j$, then
there is an edge in $\cay(G,S)$ joining $V_i\cap G$ and $V_j\cap
G$. Since $V_1,\dots,V_m$ are branch sets of an $M$-minor, there
is an edge joining $V_i$ and $V_j$. This edge must be in $\cay(G,S)$, because otherwise we can find a path in $\cay(G*H,S\cup T)$ that leaves $G$ through
one element of $G$ and returns through another element of
$G$, which is clearly impossible.
\end{proof}

\subsection{Groups with one end
}\label{S:OneEnd}

Let $\Gamma$ be a connected, locally finite graph, and denote by
$B(n,O)$ the ball of radius $n$ in $\Gamma$ centered at some fixed
vertex $O$ in $\Gamma$. The number of {\it ends} in $\Gamma$ is the
limit of the number of unbounded connected components in
$\Gamma\smallsetminus B(n,O)$ as $n\to\infty$. This limit exists either as a nonnegative integer or as $\infty$. (See \cite[pp.~144--148]{BH99} or \cite[Sections
11.4--11.6]{Mei08} for an introduction to the theory of ends.) A basic
result is that the number of
ends of a Cayley graph of a finitely generated group does not
depend on the choice of generating set (see \cite[Theorem 11.23]{Mei08}).

The following theorem is the main result of this section.

\begin{theorem}\label{T:OneEndInfInd}
Let $G$ be a finitely generated group with one end. Then there is
a finite set of generators $S$ of $G$ such that $\cay(G,S)$ is not
minor excluded.
\end{theorem}

\begin{lemma}\label{L:SufInfPaths} Assume that a one-ended group
$G$ with finite generating set $S_0$ has the property that for
each $m\in\mathbb{N}$ the Cayley graph $\cay(G,S_0)$ contains a
collection of $m$ disjoint infinite rays. Then
$G$ contains a finite generating set $S$ such that $\cay(G,S)$ is
not minor excluded.\end{lemma}

\begin{proof}
Let $\{V_i\}_{i=1}^m$ be disjoint rays in $\cay(G,S_0)$.
The sets $V_i$, after some small modifications, will become branch
sets of a $K_m$-minor in $\cay(G,S)$, where $S=S_0\cup S_0S_0\cup
S_0S_0S_0$. (Here, $S_0S_0$ and $S_0S_0S_0$ are the sets of
products of all pairs and triples, respectively, of elements in
$S_0$.) That is, we will modify $\{V_i\}$ in such a way that the
modified sets are still disjoint, each modified $V_i$ is connected
in $\cay(G,S)$, and for each $i\ne j$ there is an edge in
$\cay(G,S)$ joining a vertex of the modified $V_i$ with a vertex
of the modified $V_j$. This will be achieved in $\frac{m(m-1)}2$
steps. Intuitively speaking, each step will create the desired
``connection" between $V_i$ and $V_j$ for the pair $(i,j)$, $i\ne
j$.

Start with the pair $(1,2)$. Since the graph $\cay(G,S)$ is connected, there is a path $P$ connecting some vertex
of $V_1$ to some vertex of $V_2$. Suppose $P$ intersects another
$V_k$. We describe a procedure for ``removing" this intersection.
Write $V_k=\{y_{(k,1)},y_{(k,2)},\dots\}$, where the indexing
starts with the ray's origin and moves out towards ``infinity".
The intersection of $P$ and $V_k$ has a first vertex and a last
vertex (with respect to the indexing of the elements of $V_k$). Denote these
vertices $y_{(k,s)}$ and $y_{(k,t)}$ $(s\le t)$, respectively.

If $s=t$, that is, if $P$ intersects $V_k$ in just one
vertex, then replace $V_k$ with $V_k'=V_k\smallsetminus
\{y_{(k,s)} \}$ and leave $P$ unchanged. If $t=s+1$, replace $V_k$
with $V_k'=V_k\smallsetminus \{y_{(k,s)}, y_{(k,t)} \}$ and leave
$P$ unchanged. Then, in both cases, $V_k'$ is disjoint from $V_1$
and $V_2$, $V_k'$ is connected in $\cay(G,S)$ (since $S_0S_0$ and
$S_0S_0S_0$ are in $S$), and $V_k'$ does not intersect $P$.

If $s+1< t$, then we modify both $P$ and $V_k$, as follows. Clearly we can find two paths, $L_1$ and $L_2$, using elements of $S_0S_0$ and $S_0S_0S_0$ such that $L_1$ is a path from $y_{(k,s)}$ to $y_{(k,t)}$, $L_2$ is a path from $y_{(k,s+1)}$ to $y_{(k,t+1)}$, $L_1$ and $L_2$ belong to $V_k$, and $L_1$ is disjoint from $L_2$. Now we replace the piece of $P$ that connects $y_{(k,s)}$ to $y_{(k,t)}$ with $L_1$ to obtain a new path $P'$, and we create $V_k'$ from $V_k$ by removing all the vertices in $V_k$ from $y_{(k,s)}$ to $y_{(k,t)}$ and adding the vertices of $L_2$, i.e., $V_k'=\big(V_k\smallsetminus \{ y_{(k,i)}\,:\, s\leq i \leq t \}\big) \cup L_2$. Then, $V_k'$ is disjoint from $V_1$ and $V_2$, $V_k'$ is connected in $\cay(G,S)$, and $V_k'$ does not intersect $P'$.

The procedure used for ``removing" the intersection of $P$ with $V_k$ did not introduce any new intersections, although it is possible that some
other intersections disappeared in the process. (Since the sets $\{V_i\}$ are disjoint, $P$ cannot intersect more than one $V_k$ in the same place.)
Thus, the procedure can be repeated to ``remove" each intersection one at a time until we obtain a path from $V_1$ to $V_2$ that does not intersect any of the modified $V_i$'s.

We continue to use the method described above for constructing the desired connection between $V_1$ and $V_2$ until we establish
``connections" for every pair $(i,j)$, $i \ne j$. Specifically,
once we have established ``connections'' between some of the $V_i$ and would then like to arrange a new connection,
we remove a ball, $B$, centered at
$1$ from $\cay(G,S)$, whose radius is large enough to contain all vertices
that were used for previous paths and all pieces of $V_k$'s that
were involved in the previous modifications.
Since $G$ has one end, $\cay(G,S)\smallsetminus B$ has an
unbounded connected component, $\Theta$. Thus, we can use the method above on the infinite connected pieces of the rays $V_i$ in $\Theta$. The new modifications will not destroy
previous connections because they are made away from the previously constructed connections. After  ``connections" have been constructed for all pairs $(i,j)$, $i \ne j$,
the sets obtained from the final modification will be branch sets for
$K_m$ in $\cay(G,S)$.
\end{proof}

\begin{proof}[Proof of Theorem \ref{T:OneEndInfInd}] By Lemma
\ref{L:SufInfPaths}, it suffices to construct, for an arbitrary
$m\in\mathbb{N}$, a collection of $m$ disjoint rays in
$\cay(G,S_0)$, where $S_0\subset G$ is some finite generating set.
We construct such rays using Menger's theorem
\cite[Section 3.3]{Die00}, which states: if vertex sets $A$ and $B$ in a graph
cannot be separated by removing fewer than $k$ vertices, then
there are $k$ disjoint paths joining $A$ and $B$. In
\cite{Die00} this result is proved for finite graphs, but it also holds
for infinite locally finite graphs (and even in a more
general context; see \cite{Aha87} and \cite{Hal74}).

Begin by fixing a finite generating set $\widetilde S$ in
$G$. For each vertex $v$, let $\ell(v)$ denote the length of $v$, i.e., the distance from $v$ to the unit element 1 in $\cay(G,\widetilde S)$. Consider the following alternatives.
\begin{enumerate}
    \item[(i)] There is some $m\in\mathbb{N}$ such that for each
$R\in\mathbb{N}$ there is a set $C_R$, consisting of $m$ vertices of length at least $R$, such that the removal of $C_R$ from $\cay(G,\widetilde S)$ disconnects
$1$ from the infinite component.
    \item[(ii)] There is no such $m$.
\end{enumerate}

Consider case (ii). Let $m\in\mathbb{N}$ be given and choose $S_0=\widetilde S$. Then there is an
$R\in\mathbb{Z}$ such that the removal of $m$ vertices of length at least $R$ in
$\cay(G,S_0)$ cannot disconnect $1$ and ``infinity''.
Let $L$ be a natural number bigger than $R$.
Use Menger's theorem on a one-element set $\{w_L\}$ with
$\ell(w_L)=L>R$ and an $m$-element subset $A$ of vertices from
$\{v:~\ell(v)=R\}$. Such a subset exists since, by assumption,
$\#\{v:~\ell(v)=R\}> m$ (otherwise this set would be an $m$-element
set disconnecting $1$ and ``infinity''). This yields a sequence
$\{(P_1^L,\dots,P_m^L)\}_{L=R+1}^\infty$ of $m$-tuples of disjoint
paths joining $A$ and $w_L$. Since $\cay(G,S_0)$ is
locally finite, we can find a convergent subsequence in
$\{P_1^L\}_{L=R+1}^\infty$. Let $I_1$ be the corresponding set of
indices and $P_1$ be the limiting path. Consider the sequence
$\{P_2^L\}_{L\in I_1}$. It contains a convergent subsequence. It
is easy to check that the limit $P_2$ of this subsequence is
disjoint from $P_1$. Continue in the obvious way to obtain $m$ disjoint rays in $\cay(G,S_0)$.

Now consider case (i). In this case, by Lemma~\ref{L:CutsInfOrd} below, the group $G$ contains an element of infinite order. The subgroup
$H$ generated by an element $v$ of infinite order must have infinite index, since $G$ has one end and a  group with an infinite cyclic subgroup of finite index has two ends (see~\cite[Corollary 11.34]{Mei08}).
Thus, choosing $S_0=\widetilde
S\cup\{v\}$ yields an infinite collection of
disjoint rays in $\cay(G,S_0)$, $\big\{\{vg_i,v^2g_i,v^3g_i,\dots\}\,:\,i\in\mathbb{N}\big\}$, where
$\{g_i\,:\,i\in\mathbb{N}\}$ are right coset representatives of $G/H$.
\end{proof}

\begin{lemma}\label{L:CutsInfOrd} Let $G$ be a one-ended group
with finite generating set $\widetilde
S$. Assume that there exists an $m\in\mathbb{N}$ such that for each
$R\in\mathbb{N}$ there is a set $C_R$, consisting of $m$ vertices
of length at least $R$ in $\cay(G,\widetilde
S)$, such that the removal of
$C_R$ disconnects $1$ from ``infinity''. Then the group $G$
contains an element of infinite order.
\end{lemma}

\begin{proof} Notice that the order of an element $z$ in $G$ for which the sequence $\{d_G(z^n,1)\}_{n=1}^\infty$ is unbounded must be infinite. We use some ideas from the proof of the result on groups of linear growth in \cite{Jus71,WV84} (see also \cite{IS87,Man12}) to produce such a $z$.

Assume that $\widetilde
S$ is symmetric. Fix an order on the elements of
$\widetilde S$ and then order all words of the same length in $\widetilde S$
using a lexicographic order. For each element $g$ of $G$, we call the first word in this order
among the shortest words representing $g$
a {\it distinguished word}. It is easy to see that
a subword of a distinguished word is distinguished.

Let $w_1,\dots,w_m$ be the distinguished words representing
elements of the set $C_R$ for $R=m$. Let $u_1,\dots,u_m$ be the
starting pieces of length $m$ of $w_1,\dots,w_m$.
By taking a pointwise limit of
distinguished words with indefinitely increasing lengths,
we can find a geodesic ray $P$
starting at $1$ in $G$ such that each finite piece of it is a distinguished word.
Because $1$ and ``infinity'' are disconnected in $\cay(G,\widetilde S) \smallsetminus C_R$, $P$ must pass through $C_R$. Furthermore, thinking of $P$ as an infinitely long  distinguished word, each subword of $P$ of length $m$ is one of the words $u_1,\dots,u_m$.
Thus, $P$ has at most $m$ distinct subwords of length $m$. Therefore, by \cite[Theorem 3.3]{Man12}, each subword $w$ of $P$ of length at least $2m$ is of the form $w=utv$, where $t$ is $p$-periodic (i.e., $t=z^kr$, where $k$ is some exponent, $z$ is a word of length $p$ and $r$ is a word of length at most $p$) for some $0<p\le m$, and $u$ and $v$ have lengths at most $m-p$. Apply this result to each of the initial subwords $w$ of $P$ of length at least $2m$. It is clear that, when written in the form $w=utv=uz^krv$, infinitely many of the $w$ will have the same $z$.
Since the lengths of all of the $u$'s, $v$'s, and $r$'s are bounded above by $m$, it follows that the sequence $\{d_G(z^n,1)\}_{n=1}^\infty$ is unbounded.
\end{proof}

\begin{remark}
Recall that an infinite finitely generated group has one, two, or infinitely
many ends (see~\cite[Theorem 8.32, p.~146]{BH99}). Furthermore, a group with two ends contain $\mathbb{Z}$ as a subgroup of finite index.
Hence, the question of whether or not a given group has a finite generating set for which the corresponding Cayley graph is not minor excluded is answered for one-ended groups and two-ended groups in Sections~\ref{S:OneEnd} and~\ref{S:VirtFree}, respectively.
The only case that remains open is the case of infinitely many ends.
\end{remark}

\section{Open problems}

The following two problems are, in our opinion, the most
intriguing open problems related to this paper.

\begin{problem} Let $G$ be a group of asymptotic dimension at least $3$.
Does it follow that $\cay(G,S)$ is not minor excluded for any
choice of generating set $S$?
\end{problem}

\begin{problem} Let $G$ be a group that is not virtually free.
Does it follow that $\cay(G,S)$ is not minor excluded for some
choice of generating set $S$?
\end{problem}


\begin{small}

\end{small}

\end{document}